\documentclass[12pt,a4paper]{article}
\usepackage{indentfirst}
\usepackage{amsmath,amssymb,amsfonts,amsthm,graphics}
\usepackage{makeidx}
\usepackage{color}

\newtheorem{theorem}{Theorem}

\newtheorem{proposition}{Proposition}
\newtheorem{lemma}{Lemma}

\newtheorem{observation}{Observation}

\begin{document}
\title{\Large\bf Graphs with large generalized $3$-connectivity \footnote{Supported by NSFC
No.11071130.}}
\author{\small Hengzhe Li, Xueliang Li, Yaping Mao, Yuefang Sun
\\
\small Center for Combinatorics and LPMC-TJKLC
\\
\small Nankai University, Tianjin 300071, China
\\
\small lhz2010@mail.nankai.edu.cn; lxl@nankai.edu.cn;\\
\small maoyaping@ymail.com; bruceseun@gmail.com}
\date{}
\maketitle
\begin{abstract}
Let $S$ be a nonempty set of vertices of a connected graph $G$. A collection $T_1,\cdots,T_{\ell}$ of trees in $G$
is said to be internally disjoint trees connecting $S$ if $E(T_i)\cap E(T_j)=\emptyset$ and $V(T_i)\cap V(T_j)=S$ for
any pair of distinct integers $i,j$, where $1\leq i,j\leq r$. For an integer $k$ with $2\leq k\leq n$, the generalized
$k$-connectivity $\kappa_k(G)$ of $G$ is the greatest positive integer $r$ such that $G$ contains at least $r$ internally
disjoint trees connecting $S$ for any set $S$ of $k$ vertices of $G$. Obviously, $\kappa_2(G)$ is the connectivity
of $G$. In this paper, sharp upper and lower bounds of $\kappa_3(G)$ are given for a connected graph $G$ of order $n$,
that is, $1\leq \kappa_3(G)\leq n-2$. Graphs of order $n$ such that $\kappa_3(G)=n-2,\, n-3$ are characterized, respectively.

{\flushleft\bf Keywords}: connectivity, internally disjoint trees, generalized connectivity.\\[2mm]
{\bf AMS subject classification 2010:} 05C40, 05C05.
\end{abstract}

\section{Introduction}

All graphs in this paper are undirected, finite and simple. We refer to book \cite{bondy} for graph theoretical
notation and terminology not described here.

The generalized connectivity of a graph $G$, which was introduced by Chartrand et al. in \cite{Chartrand1},
is a natural and nice generalization of the concept of connectivity. A tree $T$ is called \emph{an $S$-tree} if
$S\subseteq V(T)$, where $S\in V(G)$. A collection $T_1,\cdots,T_{\ell}$ of trees in $G$ is said to be
\emph{internally disjoint trees connecting $S$} if $E(T_i)\cap E(T_j)=\emptyset$ and $V(T_i)\cap V(T_j)=S$ for
any pair of distinct integers $i,j$, where $1\leq i,j\leq r$. For an integer $k$ with $2\leq k\leq n$,
the \emph{generalized $k$-connectivity} $\kappa_k(G)$ of $G$ is the greatest positive integer $r$ such that
$G$ contains at least $r$ internally disjoint trees connecting $S$ for any set $S$ of $k$ vertices of $G$.
Obviously, $\kappa_2(G)$ is the connectivity of $G$. By convention, for a connected graph with less than $k$ 
vertices, we set $\kappa_k(G)=1$; for a disconnected graph $G$, we set $\kappa_k(G)=0$.

In addition to being natural combinatorial measures, the generalized connectivity can be motivated by
their interesting interpretation in practice. For example, suppose that $G$ represents a network. If one
considers to connect a pair of vertices of $G$, then a path is used to connect them. However, if one wants
to connect a set $S$ of vertices of $G$ with $|S|\geq 3$, then a tree has to be used to connect them.
This kind of tree with minimum order for connecting a set of vertices is usually called a Steiner tree,
and popularly used in the physical design of VLSI, see \cite{Sherwani}. Usually, one wants to consider
how tough a network can be, for the connection of a set of vertices. Then, the number of totally independent
ways to connect them is a measure for this purpose. The generalized $k$-connectivity can serve for
measuring the capability of a network $G$ to connect any $k$ vertices in $G$.

There have appeared many results on the generalized connectivity, see \cite{Chartrand1, Chartrand2, 
Okamoto, Li1, Li2, Li3, Li4, Li5}.
Chartrand et al. in \cite{Chartrand2} obtained the following result in the generalized connectivity.
\begin{lemma}\cite{Chartrand2}
For every two integers $n$ and $k$ with $2\leq k\leq n$,
$$
\kappa_k(K_n)=n-\lceil k/2\rceil.
$$
\end{lemma}

The following result is given by Li et al. in \cite{Li4}, which will be used later.
\begin{lemma}\cite{Li4}
For any connected graph $G$, $\kappa_3(G)\leq \kappa(G)$.
Moreover, the upper bound is sharp.
\end{lemma}

In Section 2, sharp upper and lower bounds of $\kappa_3(G)$ are given for a connected graph $G$ of order $n$,
that is, $1\leq \kappa_3(G)\leq n-2$. Moreover, graphs of order $n$ such that $\kappa_3(G)=n-2,\, n-3$ are
characterized, respectively.

\section{Graphs with $3$-connectivity $n-2, n-3$}

For a graph $G$, let $V(G)$, $E(G)$ be the set of vertices, the set of edges, respectively, and $|G|$ and
$\|G\|$ the order, the size of $G$, respectively. If $S$ is a subset of vertices of a graph $G$,
the subgraph of $G$ induced by $S$ is denoted by $G[S]$. If $M$ is a subset of edges of $G$, the subgraph
of $G$ induced by $M$ is denoted by $G[M]$. As usual, the \emph{union} of two graphs $G$ and $H$ is the
graph, denoted by $G\cup H$, with vertex set $V(G)\cup V(H)$ and edge set $E(G)\cup E(H)$. Let $mH$ be
the disjoint union of $m$ copies of a graph $H$. For $U\subseteq V(G)$, we denote $G\setminus U$ the subgraph
by deleting the vertices of $U$ along with the incident edges from $G$. Let $d_G(v)$, simply denoted by $d(v)$,
be the degree of a vertex $v$, and let $N_G(v)$ be the neighborhood set of $v$ in $G$. A subset $M$ of $E(G)$
is called a \emph{matching} in $G$ if its elements are such edges that no two of them are adjacent in $G$.
A matching $M$ saturates a vertex $v$, or $v$ is said to be \emph{$M$-saturated}, if some edge of $M$ is
incident with $v$; otherwise, $v$ is \emph{$M$-unsaturated}. $M$
is a \emph{maximum matching} if $G$ has no matching $M'$ with $|M'|>|M|$.

\begin{observation}
If $G$ is a graph obtained from the complete graph $K_n$ by deleting an edge set $M$ and $\Delta(K_n[M])\geq 3$,
then $\kappa_3(G)\leq n-4$.
\end{observation}

The observation above indicates that if $\kappa_3(G)\geq n-3$, then each component of $K_n[M]$ must be a path or a cycle.

After the preparation above, we start to give our main results of this paper. At first, we give the bounds of $\kappa_3(G)$.
\begin{proposition}
For a connected graph $G$ of order $n \ (n\geq 3)$, $1\leq \kappa_3(G)\leq n-2$. Moreover, the upper and lower bounds are sharp.
\end{proposition}

\begin{proof}
It is easy to see that $\kappa_3(G)\leq \kappa_3(K_n)$. From this together with Lemma 1, we have $\kappa_3(G)\leq n-2$.
Since $G$ is connected, $\kappa_3(G)\geq 1$. The result holds.

It is easy to check that the complete graph $K_n$ attains the upper bound and the complete
bipartite graph $K_{1,n-1}$ attains the lower bound.
\end{proof}

\begin{theorem}
For a connected graph $G$ of order $n$, $\kappa_3(G)=n-2$ if and only if $G=K_n$ or $G=K_n\setminus e$.
\end{theorem}

\begin{proof}
\emph{Necessity}
If $G=K_n$, then we have $\kappa_3(G)=n-2$ by Lemma 1. If $G=K_n\setminus e$, it follows by Proposition 1 that
$\kappa_3(G)\leq n-2$. We will show that $\kappa_3(G)\geq n-2$. It suffices to show that for any $S\subseteq V(G)$ such that $|S|=2$,
there exist $n-2$ internally disjoint $S$-trees in $G$.

Let $e=uv$, and $W=G\setminus\{u,v\}=\{w_1,w_2,\cdots,w_{n-2}\}$. Clearly, $G[W]$ is a complete graph of order $n-2$.

\begin{figure}[h,t,b,p]
\begin{center}
\scalebox{0.8}[0.8]{\includegraphics{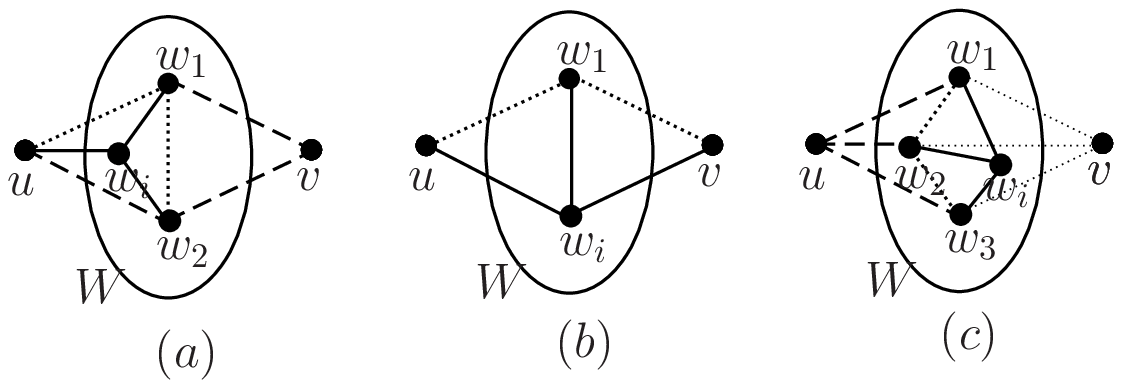}}\\
Figure 1 The edges of a tree are by the same type of lines.
\end{center}
\end{figure}

If $|\{u,v\}\cap S|=1$ (See Figure 1 $(a)$),
without loss of generality, let $S=\{u,w_1,w_2\}$. The trees $T_i=
w_iu\cup w_iw_1\cup w_iw_2$ together with $T_1=uw_1\cup w_1w_2$, $T_2=uw_2\cup vw_2\cup vw_1$
form $n-2$ pairwise internally disjoint $S$-trees, where $i=2,\cdots,n-2$.

If $|\{u,v\}\cap S|=2$(See Figure 1 $(b)$),
without loss of generality, let $S=\{u,v,w_1\}$. The trees $T_i=
w_iu\cup w_iv\cup w_iw_1$ together with $T_1=uw_1\cup w_1v$ form $n-2$ pairwise internally disjoint $S$-trees, where $i=2,\cdots,n-2$.

Otherwise, suppose $S\subseteq W$ (See Figure 1 $(c)$). Without loss of generality, let $S=\{w_1,w_2,w_3\}$. The
trees $T_i=w_iw_1\cup w_iw_2\cup w_iw_3(i=4,5,\cdots,n-2)$ together with $T_1=w_2w_1\cup w_2w_3$ and $T_2=uw_1\cup uw_2\cup uw_3$
and $T_3=vw_1\cup vw_2\cup vw_3$ form $n-2$ pairwise internally disjoint $S$-trees.

From the arguments above , we conclude that $\kappa_3(K_n\setminus e)\geq n-2$. From this together with Proposition 1,
$\kappa(K_n\setminus e)=n-2$.

\emph{Sufficiency}
Next we show that if $G\neq K_n, K_n\setminus e$, then $\kappa_3(G)\leq n-3$, where $G$ is a connected graph.
Let $G$ be the graph obtained from $K_n$ by deleting two edges. It suffices to prove that $\kappa_3(G)\leq n-3$.
Let $G=K_n\setminus\{e_1,e_2\}$, where $e_1,e_2\in E(K_n)$. If $e_1$ and $e_2$ has a common vertex and form a $P_3$, 
denoted by $v_1,v_2,v_3$. Thus $d_G(v)=n-3$. So $\kappa_3(G)\leq \delta(G)\leq n-3$.
If $e_1$ and $e_2$ are independent edges. Let $e_1=xy$ and $e_2=vw$. Let $S=\{x,y,v\}$.
We consider the internally disjoint $S$-trees.
It is easy to see that
$d_{G}(x)=d_{G}(y)=d_{G}(v)=n-2$. Furthermore, each edge incident to $x$ (each neighbor adjacent to $x$)
in $G$ belongs to an $S$-tree so that we can obtain $n-2$ $S$-trees. The same is true for the vertices $y$ and $v$.
Let $\mathcal {T}$ be a set of internally disjoint $S$-trees that contains as many $S$-trees as possible
and $U=N_G(x)\cap N_G(y)\cap N_G(v)$. There exist at most $|U|=n-4$ $S$-tree in $\mathcal {T}$ that contain at 
least one vertex in $U$. Next we show that there exist one $S$-tree in $G\setminus U$. Suppose that there exist two internally disjoint
$S$-trees in $G\setminus U$. Since $G\setminus U$ is cycle of order $4$, and there exists at most one $S$-tree
in $G\setminus U$. So $\kappa_3(G)=|\mathcal {T}|\leq n-3$.
\end{proof}

\begin{theorem}
Let $G$ be a connected graph of order $n(n\geq 3)$. $\kappa_3(G)=n-3$ if and only if $G$ is a graph obtained from the
complete graph $K_n$ by deleting an edge set $M$ such that $K_n[M]=P_4$ or $K_n[M]=P_3\cup P_2$ or $K_n[M]=C_3\cup P_2$
or $K_n[M]=r P_2( 2\leq r\leq \lfloor\frac{n}{2}\rfloor)$.
\end{theorem}

\begin{proof}
\emph{Sufficiency.}~~Assume that $\kappa_3(G)=n-3$. Then $|M|\geq 2$ by Theorem 1 and each component of $K_n[M]$
is a path or a cycle by Observation 1. We will show that the following claims hold.

\textbf{Claim 1.}~~$K_n[M]$ has at most one component of order larger than 2.

Suppose, to the contrary, that $K_n[M]$ has two components of order larger than 2, denoted by $H_1$ and $H_2$ 
(See Figure 2 $(a)$). Pick a set $S=\{x,y,z\}$ such that $x,y\in H_1$, $z\in H_2$,
$d_{H_1}(y)=d_{H_2}(z)=2$, and $x$ is adjacent to $y$ in $H_1$. Since $d_G(y)=n-1-d_{H_1}(y)=n-3$, each edge incident
to $y$ (each neighbor adjacent to $y$) in $G$ belongs to an $S$-tree so that we can obtain $n-3$ internally disjoint 
$S$-trees. The same is true for the vertex $z$. The same is true for the vertices $y$ and $v$.
Let $\mathcal {T}$ be a set of internally disjoint $S$-trees that contains as many $S$-trees as possible and $U$ be the
vertex set whose elements are adjacent to both of $y$ and $z$. There exist at most $|U|=n-6$ $S$-trees in $\mathcal {T}$
that contain a vertex in $U$.

Next we show that there exist at most $2$ $S$-trees in $G\setminus U$ (See Figure 2 $(a)$). Suppose that there exist
$3$ internally disjoint $S$-trees in $G\setminus U$. Since $d_{G\setminus U}(y)=d_{G\setminus U}(z)=3$, $yz$ must be
in an $S$-tree, say $T_{n-5}$. Then we must use one element of the edge set $E_1=\{zx,v_2z,v_3y,v_1y\}$ if we want to
reach $x$ in $T_{n-5}$. Thus $d_{T_{n-5}}(y)=2$ or $d_{T_{n-5}}(z)=2$, which implies that there exists at most one
$S$-tree except $T_{n-5}$ in $G\setminus U$. So $\kappa_3(G)=|\mathcal {T}|\leq n-4$, a contradiction.

\begin{figure}[h,t,b,p]
\begin{center}
\scalebox{0.8}[0.8]{\includegraphics{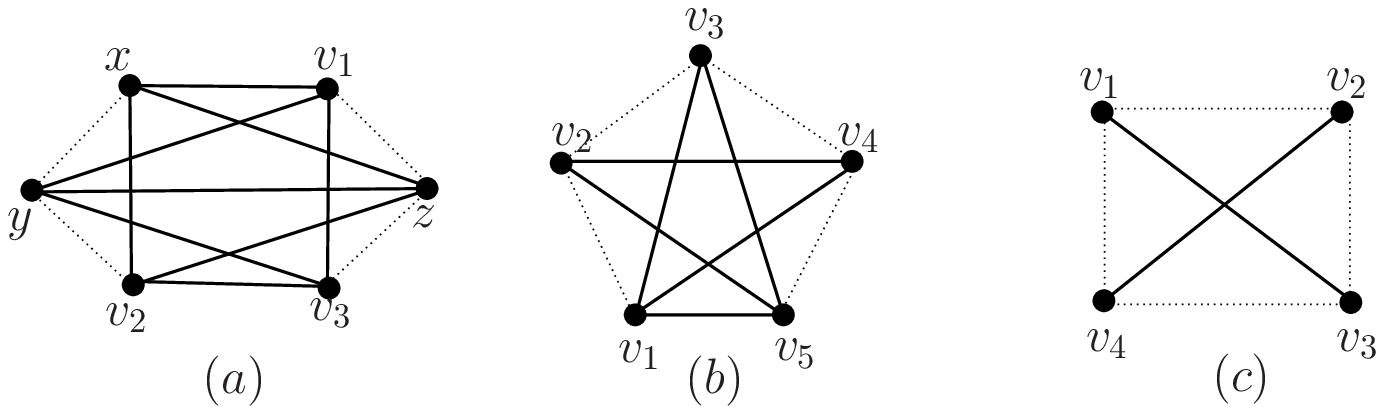}}\\
Figure 2 Graphs for Claim 1 and Claim 2(The dotted lines stand for edges in $M$).
\end{center}
\end{figure}

\textbf{Claim 2.}~~If $H$ is a component of $K_n[M]$ of order larger than three, then $K_n[M]=P_4$.

Suppose, to the contrary, that $H$ is a path or a cycle of order larger than $4$, or a cycle of order $4$,
or $H$ is a path of order $4$ and $K_n[M]$ has another component.

If $H$ is a path or a cycle of order larger than $4$, we can pick a $P_5$ in $H$. Let $P_5=v_1,v_2,v_3,v_4,v_5$(See Figure 2 $(b)$)
and $S=\{v_2,v_3,v_4\}$. Since $d_H(v_2)=d_H(v_3)=d_H(v_4)=2$,
$d_{G}(v_2)=d_{G}(v_3)=d_{G}(v_4)=n-3$. Furthermore, each edge incident to $v_2$ (each neighbor adjacent to $v_2$)
in $G$ belongs to an $S$-tree so that we can obtain $n-3$ $S$-trees. The same is true for the vertices $y$ and $z$.
Let $\mathcal {T}$ be a set of internally disjoint $S$-trees that contains as many $S$-trees as possible and $U=N_G(v_2)\cap N_G(v_3)\cap N_G(v_4)$. There
exist at most $|U|=n-5$ $S$-tree in $\mathcal {T}$ that contain at least one vertex in $U$. Next we show that there
exist at most one $S$-tree in $G\setminus U$ (See Figure 2 $(b)$). Suppose that there exist two internally disjoint
$S$-trees in $G\setminus U$. Since $d_{G\setminus U}(v_2)=d_{G\setminus U}(v_4)=2$, $v_2v_4$ must be in an $S$-tree,
say $T_{n-5}$. Then we must use one element of $\{v_1,v_5\}$ if we want to reach $v_3$ in $T_{n-5}$. This implies
that there exists at most one $S$-tree except $T_{n-5}$ in $G\setminus U$. So $\kappa_3(G)=|\mathcal {T}|\leq n-4$, a contradiction.

If $H$ is a cycle of order $4$, let $H=v_1,v_2,v_3,v_4$(See Figure 2 $(c)$), and $S=\{v_1,v_2,v_3\}$.
Since $d_H(v_1)=d_H(v_2)=d_H(v_3)=2$, $d_{G}(v_1)=d_{G}(v_2)=d_{G}(v_3)=n-3$. Furthermore, each edge incident to $v_1$ in $G$ belongs to an $S$-tree so that we can obtain $n-3$ $S$-trees. The same is true for
the vertices $v_2$ and $v_3$. Let $\mathcal {T}$ be a set of internally disjoint $S$-trees that contains as many $S$-trees as possible and
$U=N_G(v_2)\cap N_G(v_3)\cap N_G(v_4)$. There exist at most $|U|=n-4$ $S$-trees in $\mathcal {T}$ that contain at
least one vertex in $U$. It is obvious that $G\setminus U$ is disconnected, and we will show that there exists no $S$-tree
in $G\setminus U$(See Figure 2 $(c)$). So $\kappa_3(G)=|\mathcal {T}|\leq n-4$, a contradiction. .

Otherwise, $H$ is a path order $4$ and $K_n[M]$ has another component. By Claim 1, the component must be an edge, denoted by
$P_2=u_1u_2$. Let $H=P_4=v_1,v_2,v_3,v_4$(See Figure 3 $(a)$) and $S=\{v_2,v_3,u_1\}$. Since $d_H(v_2)=d_H(v_3)=2$, we have $d_{G}(v_2)=d_{G}(v_3)=n-3$. Furthermore, each edge incident to $v_2$ (each neighbor adjacent to $v_2$) in $G$ belongs to an $S$-tree so that we can obtain $n-3$ $S$-trees. The same is true for the vertex $v_3$. Let $\mathcal {T}$ be a set of internally disjoint $S$-trees that contains as many $S$-trees as possible and $U$ be the vertex set whose elements are adjacent to both of $v_2$, $v_3$ and $u_1$. There exist at most $|U|=n-6$ $S$-trees in $\mathcal {T}$ that contain at least one vertex in $U$. Next we show that there exist
at most two $S$-trees in $G\setminus U$. Suppose that there exist $3$ internally disjoint $S$-trees in $G\setminus U$.
Since $d_{G\setminus U}(v_2)=d_{G\setminus U}(v_3)=3$, each edge incident to $v_2$ (each neighbor adjacent to $v_2$) in $G$ belongs to an
$S$-tree so that we can obtain $3$ $S$-trees. The same is true for the vertex $v_3$. This implies that $v_2u_2$ belongs to an $S$-trees,
denoted by $T_1$, and $v_3u_2$ belongs to an $S$-trees, denoted by $T_2$. Clearly, $T_1=T_2$. Otherwise, $u_2\in T_1\cap T_2$, which
contradicts to that $T_1$ and $T_2$ are internally disjoint $S$-trees. Then $v_2u_2,v_3u_2\in E(T_{1})$.
If we want to form $T_{1}$, we need the vertex $v_1$ or $v_4$. Without loss of generality, let $v_1\in V(T_{1})$.
It is easy to see that there exists exactly one $S$-tree except $T_{1}$ in $G\setminus U$ (See Figure 3 $(b)$), which implies
that $\kappa_3(G)\leq n-4$. So $\kappa_3(G)=|\mathcal {T}|\leq n-4$, a contradiction.

\textbf{Claim 3.}~~If $H$ is a component of $K_n[M]$ of order $3$, then $K_n[M]=C_3\cup P_2$ or $K_n[M]=P_3\cup P_2$.

By the similar arguments to the claims above, we can deduce the claim.

\begin{figure}[h,t,b,p]
\begin{center}
\scalebox{0.8}[0.8]{\includegraphics{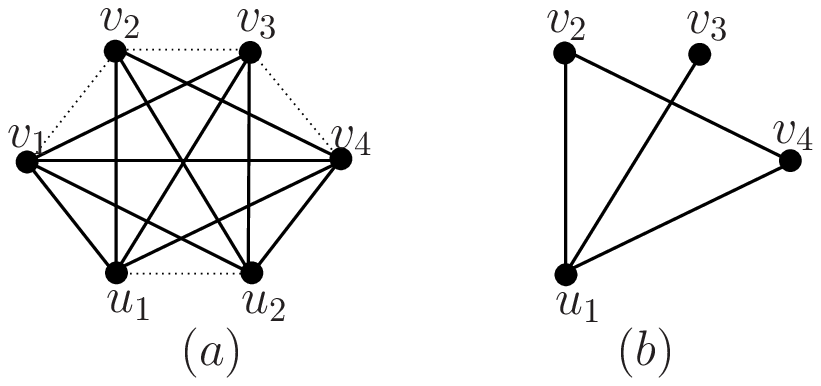}}\\
Figure 3 Graphs for Claim 2(The dotted lines stands for edges in $M$).
\end{center}
\end{figure}

From the arguments above, we can conclude that $G$ is a graph obtained from the complete graph $K_n$ by deleting an edge set $M$ such
that $K_n[M]=P_4$ or $K_n[M]=P_3\cup P_2$ or $K_n[M]=C_3\cup P_2$ or $K_n[M]=r P_2(2\leq r\leq \lfloor\frac{n}{2}\rfloor)$.

\emph{Necessity.}  We show that $\kappa_3(G)\geq n-3$ if $G$ is a graph obtained from the complete graph $K_n$ by deleting an edge
set $M$ such that $K_n[M]=P_4$ or $K_n[M]=P_3\cup P_2$ or $K_n[M]=C_3\cup P_2$ or $K_n[M]=r P_2( 2\leq r\leq \lfloor\frac{n}{2}\rfloor)$.
We consider the following cases:

\textbf{Case 1.}~~$K_n[M]=r P_2(2\leq r\leq \lfloor\frac{n}{2}\rfloor)$.

In this case, $M$ is a matching of $K_n$. We only need to prove that $\kappa_3(G)\geq n-3$ when $M$
is a maximum matching of $K_n$. Let $S=\{x,y,z\}$. Since $|S|=3$, $S$ contains at most a pair of adjacent
vertices under $M$.

If $S$ contains a pair of adjacent vertices under $M$, denoted by $x$ and $y$, then the
trees $T_i=w_ix\cup w_iy\cup w_iz$ together with $T_1=xy\cup yz$ form $n-3$ pairwise internally disjoint trees
connecting $S$, where $\{w_1,w_2,\cdots,w_{n-4}\}=V(G)\setminus \{x,y,z,z'\}$ such that $z'$ is the adjacent vertex
of $z$ under $M$ if $z$ is $M$-saturated, or $z'$ is any vertex in $V(G)\setminus \{x,y,z\}$ if $z$ is $M$-unsaturated.
If $S$ contains no pair of adjacent vertices under $M$, then the trees $T_i=w_ix\cup w_iy\cup w_iz$
together with $T_1=yx\cup xy'\cup y'z$ and $T_2=yx'\cup zx'\cup zx$ and $T_3=zy\cup yz'\cup z'x$ form $n-3$ pairwise
edge-disjoint $S$-trees, where $\{w_1,w_2,\cdots,w_{n-6}\}=V(G)\setminus \{x,y,z,x',y',z'\}$, $x',y',z'$ are the
adjacent vertices of $x,y,z$ under $M$, respectively, if $x,y,z$ are all $M$-saturated, or one of $x',y',z'$ is any
vertex in $V(G)\setminus \{x,y,z\}$ if the vertex is $M$-unsaturated.

From the arguments above , we know that $\kappa(S)\geq n-3$ for $S\subseteq V(G)$. Thus $\kappa_3(G)\geq n-3$.
From this together with Theorem 1, we know $\kappa_3(G)=n-3$.

\textbf{Case 2.}~~$K_n[M]=C_3\cup P_2$ or $K_n[M]=P_3\cup P_2$.

If $\kappa_3(G)\geq n-3$ for $K_n[M]=C_3\cup P_2$, then $\kappa_3(G)\geq n-3$ for $K_n[M]=P_3\cup P_2$.
So we only consider the former. Let $C_3=v_1,v_2,v_3$ and $P_2=u_1u_2$, and let $S=\{x,y,z\}$ be a $3$-set of $G$.
If $S=V(C_3)$, then there exist $n-3$ pairwise internally disjoint $S$-trees since each vertex in $S$ is adjacent
to each vertex in $G\setminus S$. Suppose $S\neq V(C_3)$.

If $|S\cap V(C_3)|=2$, without loss of generality, assume that $x=v_1$ and $y=v_2$. When $S\cap V(P_2)\neq \emptyset$, say $z=u_1$,
the trees $T_i=w_ix\cup w_iy\cup w_iz$ together with $T_{n-4}=xz\cup yz$ and $T_{n-3}=xu_2\cup u_2v_3\cup zv_3\cup u_2y$
form $n-3$ pairwise internally disjoint trees connecting $S$, where $\{w_1,w_2,\cdots,w_{n-5}\}=V(G)\setminus \{x,y,z,u_2,v_3\}$.
When $S\cap V(P_2)=\emptyset$, the trees $T_i=w_ix\cup w_iy\cup w_iz$ together with $T_{n-3}=xz\cup zy$ are $n-3$ pairwise
internally disjoint trees connecting $S$, where $\{w_1,w_2,\cdots,w_4\}=V(G)\setminus \{x,y,z,v_3\}$.

If $|S\cap V(C_3)|=1$, without loss of generality, assume $x=v_1$. When $|S\cap V(P_2)|=2$, say $y=u_1$ and $z=u_2$,
the trees $T_i=w_ix\cup w_iy\cup w_iz$ together with $T_{n-4}=xz\cup v_2z\cup v_2y$ and $T_{n-3}=xy\cup yv_3\cup zv_3$
form $n-3$ pairwise internally disjoint trees connecting $S$, where $\{w_1,w_2,\cdots,w_{n-5}\}=V(G)\setminus \{x,y,z,v_2,v_3\}$.
When $S\cap V(P_2)=1$, say $u_1=y$, the trees $T_i=w_ix\cup w_iy\cup w_iz$ together with $T_{n-5}=xz\cup zy$ and
$T_{n-4}=xu_2\cup u_2v_2\cup v_2y\cup v_2z$ and $T_{n-3}=xz\cup zv_3\cup v_3y$ are $n-3$ pairwise internally disjoint
trees connecting $S$, where $\{w_1,w_2,\cdots,w_{n-6}\}=V(G)\setminus \{x,y,z,v_2,v_3,u_2\}$.
When $|S\cap V(P_2)|=\emptyset$, the trees $T_i=w_ix\cup w_iy\cup w_iz$ together with $T_{n-4}=xz\cup zy$ and
$T_{n-3}=xy\cup yv_3\cup zv_3$ form $n-3$ pairwise internally disjoint $S$-trees, where
$\{w_1,w_2,\cdots,w_{n-5}\}=V(G)\setminus \{x,y,z,v_2,v_3\}$.

If $S\cap V(C_3)=\emptyset$, when $|S\cap V(P_2)|=0$ or $|S\cap V(P_2)|=2$, the trees $T_i=w_ix\cup w_iy\cup w_iz$
form $n-3$ pairwise internally disjoint $S$-trees, where $\{w_1,w_2,\cdots,w_{n-3}\}=V(G)\setminus \{x,y,z\}$.
When $S\cap V(P_2)=1$, say $u_1=x$, the trees $T_i=w_ix\cup w_iy\cup w_iz$ together with $T_{n-3}=xz\cup zy$ form
$n-3$ pairwise internally disjoint $S$-trees, where $\{w_1,w_2,\cdots,w_{n-4}\}=V(G)\setminus \{x,y,z,u_2\}$.

From the arguments above , we conclude that $\kappa(S)\geq n-3$ for $S\subseteq V(G)$. Thus $\kappa_3(G)\geq n-3$.
From this together with Theorem 1, it follows that $\kappa_3(G)=n-3$.

\textbf{Case 3.}~~$K_n[M]=P_4$.

This case can be proved by an argument similar to Cases 1 and 2.
\end{proof}

\end{document}